\documentclass[10pt]{article}
\usepackage[utf8x]{inputenc}
\usepackage{ucs}
\usepackage{amsmath}
\usepackage{tikz}
\usetikzlibrary{calc} 
\usepackage{caption} 
\usepackage{pgf}
\usepackage{pgffor}
\usepackage{color}
\usepackage{faktor}
\usepackage{hyperref}

\usepackage{amsfonts}
\usepackage{amssymb}
\usepackage{amsthm}
\usepackage[english]{babel}
\usepackage{cite}
\usepackage{algorithmic}
\usepackage[section]{algorithm}
\usepackage{ragged2e}
  
\newcommand{\N}{\mathbb{N}}
\newcommand{\Q}{\mathbb{Q}}
\newcommand{\Z}{\mathbb{Z}}

\newcommand{\Div}{\operatorname{Div}}
\renewcommand{\div}{\operatorname{div}}
\newcommand{\Gal}{\operatorname{Gal}}

\newcommand{\ord}{\operatorname{ord}}

\newcommand{\Jac}{\operatorname{Jac}}

\newcommand{\Norm}[2]{\operatorname{N}_{#1/#2}}

\newtheorem{theorem}{Theorem}[section]
\newtheorem{lemma}[theorem]{Lemma}
\newtheorem{definition}{Definition}
\newtheorem{proposition}[theorem]{Proposition}
\newtheorem{corollary}[theorem]{Corollary}
\newtheorem{example}[theorem]{Example}

\author{Max Kronberg \\ {\href{mailto:m.c.kronberg@rug.nl}{m.c.kronberg@rug.nl}}}

\title{Constructing Superelliptic Curves with non-trivial rational Torsion on their Jacobians}
\date{}
\begin{document}
\maketitle
\begin{abstract}
	In this paper, we describe the construction of superelliptic curves with a rational point of prescribed order on their jacobians. The construction is based on Hensel's Lemma and produces for a given integer $N$ a superelliptic curve of genus linear in $N$ with a rational $N$-division point on the jacobian. The method is illustrated with multiple examples.
\end{abstract}

\section{Introduction}
If $A$ is an abelian variety defined over a number field $K$, then we know by the Mordell-Weil Theorem \cite{weil28} that the set of $K$-rational points $A(K)$ is a finitely generated abelian group. Thus, we have
\begin{displaymath}
	A(K)\cong A_{tors}(K)\oplus \Z^r,
\end{displaymath}
for a finite group $A_{tors}(K)$, called the \emph{$K$-rational torsion subgroup} of $A$, and a positive integer $r$. Given a positive integer $g$, a number field $K$ and a finite group $G$, one can ask whether there exists an abelian variety $A$ of dimension $g$ defined over $K$ such that $A_{tors}(K)\cong G$. This question is hard to answer in general and only in some examples the answer is known. If $g=1$ and $K=\Q$, the theorem of Mazur \cite{mazur} classifies all occuring $\Q$-rational torsion subgroups for elliptic curves. One can refine the question about the whole rational torsion subgroup to the question of a rational point of prescribed order $N$. In general, when $g>1$ and $N$ is a positive integer, not much is known.
\begin{table}
\begin{tabular}{|c|c|c|}
\hline
Genus & Torsion Order & Reference\\ \hline
$g=2$ & $N\in\{2,3,\ldots,30,$ & \cite{elkies_web}, \cite{leprevost95}, \cite{leprevost04},\\ 
&$32,\ldots,36,39,40\}$ & \cite{leprevost93}, \cite{leprevost91}, \cite{ogawa1994}, \\
&&\cite{leprevost93_2},\cite{leprevost91_1},\cite{flynn90},\\
&&\cite{platonov14}\\
\hline
$g=2$ & $N\in\{45,60,63,70\}$ & \cite{howeleprevostpoonen00}, \cite{howe14} \\ \hline
$g=3$ & Subgroup up to order $864$ & \cite{howeleprevostpoonen00} \\ \hline
$g=g$ & $N=2g^2+2g+1$ & \cite{pattersonetal08},\cite{flynn91},\cite{leprevost1992}\\
&$N=2g^2+3g+1$&\\\hline
\end{tabular}

\caption{Known examples of points of finite order $N$ on jacobians of genus $g$ curves.}
\end{table}

In this paper we consider the following problem. Given a positive integer $N$, we construct a smooth projective curve defined over $\Q$ such that the \emph{jacobian} $\Jac(C)$ of $C$ has a $\Q$-rational $N$-division point, that is, a point $D\in\Jac(C)(\Q)$ with $ND=\mathcal O$, where $\mathcal O$ is the identity element of $\Jac(C)$. This paper extends some considerations of the  PhD thesis of the author \cite{kronberg16}. First, we want to fix some notation and definitions.

Throughout this paper we consider smooth projective curves $C$ but we will only write down the equations for an affine patch of the curve. Given a curve $C$ defined over a perfect field $K$ we can consider its \emph{jacobian} via the isomorphism $\Jac(C)\cong\faktor{\Div^0(C)}{\mathcal P}$ as the set of degree zero divisors modulo principal divisors on $C$. An element $D\in\Jac(C)$ is called $K$-rational if it is invariant under the action of the absolute Galois group $\Gal(\overline K/K)$, where $\overline K$ is a fixed algebraic closure of $K$, i.e. for every $D_0\in D$ and $\sigma\in \Gal(\overline K/K)$ we have that $D_0^\sigma$ is equivalent to $D_0$. Let $f\in K(C)$ be a function such that there exists a divisor $D\in \Div_0(C)$ and an integer $N$ with $\div(f)=ND$. Then the class of $D$ in $\Jac(C)$ is a $K$-rational $N$-division point. Our goal is for a given $N>1$ to construct curves $C$ defined over $\Q$ with a function $f\in\Q(C)$ having this property. By being careful in the construction we assure that the divisor $D$ is not a principal divisor, i.e. the function $f$ is not an $N$-th power. By this method we can construct curves $C$ defined over a perfect field $K$ of genus $g$ with a point of order $N$, where $N$ is linear in terms of $g$.

\begin{definition}
	Let $K$ be a perfect field and let $F\in K[X]$ be a separable polynomial with $n:=\deg(F)\geq 5$. A curve $C$ defined by
	\begin{displaymath}
		C: y^k=F(x),
	\end{displaymath}
	for some $2\leq k\in\N$ with $\gcd(k,\operatorname{char}(K))=1$ is called \emph{superelliptic curve}.
\end{definition}
\begin{proposition}
	Let $C$ be a superelliptic curve defined over $K$.
	\begin{enumerate}
		\item $C$ is a smooth affine curve.
		\item If $\gcd(k,n)=1$, then there is one point at infinity, denoted by $P_\infty$. \label{prop:item:infinity}
		\item If $k\mid n$, then there are $k$ points at infinity denoted by $P_{\infty,i}$ for $1\leq i\leq k$.\label{prop:item:infinity2}
		\item If $\gcd(k,n)=1$, then the genus of $C$ is $g(C)=\frac 1 2 (k-1)(n-1)$.\label{prop:item:genus}
		\item The equation order $\mathcal O_C:=\faktor{K[X,Y]}{(Y^k-F)}$ is integrally closed in $K(C)$.\label{prop:item:closure}
		\item $\Norm{K(C)}{K(x)}(a(x)+b(x)y)=a(x)^k+(-1)^{k+1} b(x)^kF\in K(x)$.
	\end{enumerate}
\end{proposition}
\begin{proof}
	The first statement of the proposition is obvious due to the fact that $F$ is a separable polynomial. Assertions (\ref{prop:item:infinity}),(\ref{prop:item:infinity2}) and (\ref{prop:item:genus}) follow from \cite[Prop 3.7.3]{stichtenoth} and Assertion (\ref{prop:item:closure}) follows from \cite[Prop 3.5.12]{stichtenoth}. The last statement can be shown by direct computation of the norm.
\end{proof}

\begin{lemma}\label{lem:support}
	Let $C: y^k=F(x)$ be a superelliptic curve defined over a perfect field $K$ with $\gcd(k,\operatorname{char}(K))=1=\gcd(k,\deg(F))$ and let $\psi:=a(x)+b(x)y\in\mathcal O_C$ with $\gcd(a,b)=1$. Then $\div(\psi)=D-\deg(D)P_{\infty}$ for some effective divisor $D$ and a prime $p\in K[X]$ is in the support of $D$ if and only if $p\mid(a^k+(-1)^{k+1}b^kF)$.
\end{lemma}
\begin{proof}
	The proof can be found in \cite[Prop. 13]{galbraith02}.
\end{proof}

This lemma relates the norm of a function to a superelliptic curve since a superelliptic curve $C$ is well defined by fixing $k$ and $F$. We use this fact in the following section.

The same method is used in \cite{leprevost91_1} for the construction of hyperelliptic curves with a point of finite order in its jacobian.

\section{Solving Norm Equations and Torsion}

We now want to relate the norm of a function on a superelliptic curve to the existence of a rational torsion point on the jacobian of the curve.

\begin{lemma}
	Let $C: y^k=F(x)$ be a superelliptic curve of genus $g(C)$ defined over a perfect field $K$ and let $\psi:=a(x)+b(x)y\in \mathcal O_C$ with $\gcd(a,b)=1$ such that
	\begin{displaymath}
		\Norm{K(C)}{K(x)}(\psi) = \varepsilon u(x)^N\in K[x]
	\end{displaymath}
	with $\varepsilon\in K^*$, $N\in\N$ and $1\leq\deg(u)\leq g(C)$. Then there exists $D_0\in\Jac(C)(K)$ such that $1<\ord(D_0)\mid N$.
\end{lemma}
\begin{proof}
  Suppose $\psi:=a(x)+b(x)y\in \mathcal O_C$ such that we have $a^k+(-1)^{k+1}b^kF= \varepsilon u^N$ for some unit $\varepsilon$ and $1\leq\deg(u)\leq g(C)$. Then we apply Lemma \ref{lem:support} to obtain an effective divisor $D$ with $\deg(D)=N\deg(u)$ supported exactly at the primes dividing $u$ with multiplicity $N$. Since $u$ is a polynomial over $K$, the divisor $D$ is defined over $K$. Since $\deg(u)>0$ it follows that $D$ is not only supported at infinity. This gives $\div(\psi)=D-N\cdot\deg(u)P_\infty=N(D_0-\deg(u)P_\infty)$, where $D_0$ and $D$ have the same support and $1<\ord(D_0-\deg(u)P_\infty)\mid N$ as asserted.
\end{proof}

With this lemma we directly obtain the following proposition.

\begin{proposition}
	Let $2\leq k\in\N$ and $K$ be a perfect field with $\gcd(k,\operatorname{char}(K))=1$. Let $a,b\in K[X]$ be coprime. Assume $F:=(-1)^{k+1}\frac{a^k-u^N}{b^k}\in K[X]$ is separable for some $u\in K[X]$ with $1\leq\deg(u)\leq \frac{1}{2}(k-1)(\deg(F)-1)$. If we set 
	\begin{displaymath}
		C: y^k=F(x),
	\end{displaymath}
	then $C$ is a superelliptic curve of genus $g(C)=\frac{1}{2}(k-1)(\deg(F)-1)$ and $\Jac(C)[N](K)\neq\{\mathcal O\}$.
\end{proposition}

By this corollary we need to find polynomials $a,b,u\in K[X]$ such that $F$ as in the corollary is a polynomial. This will be done by fixing the polynomial $b\in K[X]$ and finding a polynomial $R_1\in K [X]$ that is a $k$-th power for a fixed $k$. The following lemma allows us in that case to construct a suitable polynomial $a\in K[X]$.

\begin{lemma}[Hensel's Lifting]\label{lem:hensel}
	Let $2\leq k\in\N$ and $K$ be a perfect field with $\gcd(k,\operatorname{char}(K))=1$. Let $b,R_1\in K[X]$ such that $\gcd(R_1,b)=1$. If we take $u\in K[X]$ such that $R_1^k\equiv u \pmod b$, then there exist polynomials $R_l\in K[X]$ such that
	\begin{align*}
		R_l&\equiv R_{l-1} \pmod{b^{l-1}}\\
		R_l^k&\equiv u \pmod{b^l}
	\end{align*}
	for each $2\leq l\in\Z$.
\end{lemma}
We will give a proof of this well-known lemma since the proof is constructive and the construction will be used in the examples we treat in this paper.
\begin{proof}
	We will prove this lemma by showing the $l$-th iteration of the construction. Assume $2\leq k\in\N$ and $K$ is a perfect field with $\gcd(k,\operatorname{char}(K))=1$. Let $2\leq l\in\Z$ and $b,R_1\in K[X]$ with $\gcd(R_1,b)=1$ and set $u\in K[X]$ such that $R_1^k\equiv u \pmod {b^{l-1}}$. Set $\lambda_1 := \frac{R_1^k-u}{b^{l-1}}\in K[X]$; since $\gcd(R_1,b)=1$ and $R_{l-1}\equiv R_1\pmod b$ we have $\gcd(R_{l-1},b)=1$ and $\gcd(k,\operatorname{char}(K))=1$ we have that $kR_{l-1}^{k-1}$ is invertible modulo $b$. Let $\lambda_2\in K[X]$ with $\deg(\lambda_2)<\deg(b)$ be such that $\lambda_2(kR_1^{k-1})\equiv \lambda_1 \pmod{b}$. For $R_l:= R_{l-1}-\lambda_2b^{l-1}$ we have 
	\begin{align*}
		R_l&\equiv R_{l-1} \pmod{b^{l-1}}\\
		R_l^k&\equiv u \pmod{b^l}.
	\end{align*}
\end{proof}

This lemma now allows us to find the desired polynomials $F$ and $a$.

\begin{corollary}\label{cor:hensel}
	Let $1<k,N$ be integers and $K$ a perfect field such that we have $\gcd(\operatorname{char}(K),k)=1$. Let $b,u\in K[X]$ with $\gcd(b,u)=1$ such that $u$ is a $k$-th power modulo $b$. Then for all $\varepsilon\in K^*$ there exist polynomials $a,F\in K[X]$ such that
	\begin{displaymath}
		a^k+(-1)^{k+1}b^kF = \varepsilon^k u^N.
	\end{displaymath}
\end{corollary}
	
\begin{proof}
	Let $R\in K[X]$ such that
\begin{displaymath}
	R^k\equiv u \pmod b.
\end{displaymath}
Since $\gcd(b,u)=1$ by assumption, we have $\gcd(b,R)=1$ und thus we can apply Lemma \ref{lem:hensel} to find a polynomial $R_k\in K[X]$ such that $R_k^k\equiv u \pmod{b^k}$.

Writting $\varepsilon R_k^N = qb^k +a$, with $\deg(a)<k\deg(b)$, we have
\begin{displaymath}
	a^k\equiv \varepsilon^k R_k^{kN}\equiv \varepsilon^k u^N\pmod{b^k} 
\end{displaymath}
und thus there exists a polynomial $F\in K[X]$ such that
\begin{displaymath}
	a^k+(-1)^{k+1}b^kF = \varepsilon^k u^N.
\end{displaymath}
\end{proof}

This gives us a tool to construct polynomials $F\in K[X]$ for given integers $k,N$ such that the curve defined by $C:y^k=F(x)$ has a non-trivial rational point in $\Jac(C)[N](K)$. In the next section we give examples of this construction.

\section{Examples}

First we consider the trivial case, where $b=1$ . In this case no lifting is needed.
\begin{example}
	 Let $k\geq 2$ and $a\in \Z[X]$ be a polynomial with $\gcd(a_0,k)=1$, where $a_0$ is the constant coefficient of $a$. For every odd prime $p>k\cdot\deg(a)$ we set $F_p:= a^k-x^p$. Then the curve defined by
	\begin{displaymath}
		C_{p,k}: y^k=F_p(x)
	\end{displaymath}
	 is a superelliptic curve of genus $g(C_{p,k})=\frac{1}{2}(k-1)(p-1)$ with a $\Q$-rational $p$-torsion point on its jacobian.
\end{example}

In order to verify the example observe that we have 
\begin{displaymath}
	\Norm{K(C_{p,k})}{K(x)}(a^k+y)=x^p.                                            
\end{displaymath}
Thus, we only have to check whether $F$ is a separable polynomial to prove the assertion. For this we compute $\gcd(F,F')$. The formal derivative of $F$ is given by
\begin{displaymath}
	F'= -px^{p-1}+ka^{k-1}a'.
\end{displaymath}
Choosing a prime $q$ with $q\mid k$ yields $F'\equiv -px^{p-1}\pmod q$. Since $\gcd(a_0,k)=1$ by assumption we have $F(0)\not\equiv 0 \pmod q$ and thus $\gcd(F,F')\equiv 1\pmod q$. This implies $\gcd(F,F')=1$.

\begin{example}
 For this example, let $k=3$ and set $b:=X^6 + 3X^4 + 3X^2 - X + 1$ and $u:= X$ then
\begin{displaymath}
	(X^2+1)^3\equiv u \pmod b.
\end{displaymath}
Since $b$ is irreducible over $\Q$ we have $\gcd(b,u)=1$ and we can apply Hensel's Lemma to obtain
\begin{displaymath}
	R_3^3\equiv u \pmod{b^3}
\end{displaymath}
with
\begin{align*}
	9R_3 =&X^{17} - 2X^{16} + 9X^{15} - 18X^{14} + 36X^{13} - 73X^{12} + 90X^{11}
	- 172X^{10}+ 162X^9 \\
	&    - 255X^8 + 212X^7 - 248X^6		+ 185X^5 - 161X^4 + 93X^3 - 53X^2
    + 22X + 1.
\end{align*}
Suppose
\begin{displaymath}
	F_p:=-1\frac{a_p^3-X^p}{b^3}\in\Q[X]
\end{displaymath}
is separable, where $p\geq 51$ is a prime and $a_p\in\Q[X]$ with $\deg(a_p)<18$ and
\begin{displaymath}
	a_p \equiv R_3^p \mod {b^3}.
\end{displaymath}
Then the curve given by
\begin{displaymath}
	C_p: Y^3=F_p(X)
\end{displaymath}
is a superelliptic curve of genus $p-19$ and has a $\Q$-rational point of order $p$ on its jacobian.
\end{example}

With the computer algebra system \textsc{Magma}, we have tested for all primes $51\leq p\leq 2539$ that the resulting polynomial $F_p$ is separable.

The following example comes from \cite{leprevost91_1}, where this method is carried out for hyperelliptic curves.

\begin{example}[See {\cite{leprevost91_1}}]
 The hyperelliptic curve of genus two defined by
 \begin{displaymath}
  y^2=-4x^5+(t^2+10t+1)x^4-4t(2t+1)x^3+2t^2(t+3)x^2-4t^3x+t^4=:F(x)
 \end{displaymath}
over $\Q(t)$ has a $\Q(t)$-rational point of order $13$ on its jacobian.
\end{example}
In order to see that this is true consider the polynomial $b:=X^4-3X^3+(1+2t)X^2-2tX+t^2$. Then we have
\begin{displaymath}
 \left(-\frac{1}{t}X^3+\frac{3}{t}X^2-\frac{1+t}{t}X+1\right)^2\equiv X \pmod{b}
\end{displaymath}
and we can we can find by Corollary \ref{cor:hensel} with $N=13$ the mentioned polynomial $F$.

\bibliographystyle{is-alpha}
\bibliography{elliptic}

\end{document}